\documentclass[oneside]{amsart}
\usepackage{ifpdf}
\usepackage{amsmath, amsthm, amsfonts, amssymb, amscd}
\usepackage[active]{srcltx}
\usepackage{color}
\usepackage[T1]{fontenc}

    \ifpdf

      \usepackage[pdftex]{graphicx}  

    \else

      \usepackage[dvips]{graphicx}  

      \newcommand{\href}[2]{#2}

    \fi

\newcommand{\diamup}{\mc{D}}

\DeclareMathOperator{\inter}{\rm{int}}
\DeclareMathOperator{\bd}{\partial}
\DeclareMathOperator{\cl}{cl}
\DeclareMathOperator{\diam}{\rm{diam}}

\DeclareMathOperator{\fix}{\rm{Fix}}

\newcommand{\mc}{\mathcal}

\newcommand{\ol}{\overline}

\newcommand{\til}{\widetilde}

\newcommand{\R}{\mathbb{R}}\newcommand{\N}{\mathbb{N}}

\newcommand{\T}{\mathbb{T}}

\renewcommand{\SS}{\mathbb{S}}

\newcommand{\sm}{\setminus}

\newcommand{\id}{\mathrm{Id}}

\newcommand{\ie}{i.e.\ }

\newtheorem{theorem}{Theorem} 
\newtheorem{corollary}[theorem]{Corollary}
\newtheorem{lemma}[theorem]{Lemma}

\newtheorem*{proposition*}{Proposition}

\newtheorem*{theorem*}{Theorem}
\newtheorem*{claim*}{Claim}

\theoremstyle{definition}

\theoremstyle{remark}

\AtBeginDocument{ \def\MR#1{}} 

\title[A triple boundary lemma]{A triple boundary lemma\\ for surface homeomorphisms}
\author{Andres Koropecki}
\address{Universidade Federal Fluminense, Instituto de Matem\'atica e Estat\'\i stica, Rua M\'ario Santos Braga S/N, 24020-140 Niteroi, RJ, Brasil}
\email{ak@id.uff.br}
\author{Patrice Le Calvez}
\address{Sorbonne Universit\'e, Universit\'e Paris Diderot, CNRS, Institut de Math\'ematiques de Jussieu-Paris Rive Gauche, IMJ-PRG, F-75005, Paris, France}
\email{patrice.le-calvez@imj-prg.fr}

\author{Fabio Armando Tal}	\address{Instituto de Matem\'atica e Estat\'\i stica, Universidade de S\~ao Paulo, Rua do Mat\~ao 1010, Cidade Universit\'aria, 05508-090 S\~ao Paulo, SP, Brazil}
\email{fabiotal@ime.usp.br}

\begin{document}

\begin{abstract} Given an orientation-preserving and area-preserving homeomorphism $f$ of the sphere, we prove that every point which is in the common boundary of three pairwise disjoint invariant open topological disks must be a fixed point. As an application, if $K$ is an invariant Wada type continuum, then $f^n|_K$ is the identity for some $n>0$. Another application is an elementary proof of the fact that invariant disks for a nonwandering homeomorphism homotopic to the identity in an arbitrary surface are homotopically bounded if the fixed point set is inessential. The main results in this article are self-contained.
\end{abstract}

\maketitle

The aim of this short note is to prove the following result:

\begin{theorem} \label{th:3bd-free}
If an orientation-preserving homeomorphism $f\colon \SS^2\to \SS^2$ has no wandering points and $B$ is a closed topological disk which intersects three pairwise disjoint $f$-invariant open topological disks, then $f(B)\cap B\neq \emptyset$.
\end{theorem}

As an immediate consequence one has the following:
\begin{corollary}
Under the same hypotheses, every point which is in the boundary of three pairwise disjoint open invariant topological disks must be a fixed point of $f$.
\end{corollary}

Since area-preserving homeomorphisms have no wandering points, these result apply to any area-preserving homeomorphism. As an example application, we state the following: 

\begin{corollary}
If a continuum $\Lambda \subset \SS^2$ is the boundary of three different connected components of its complement and $f$ is an area-preserving homeomorphism such that $f(\Lambda)=\Lambda$, then there is $n\in \N$ such that $f^n|_{\Lambda} = \id$.
\end{corollary}
\begin{proof}
The fact that $f$ is area-preserving implies that each connected component of $\SS^2\sm \Lambda$ is periodic, so one may choose $n$ such that $f^n$ is orientation-preserving and leaves invariant three connected components of $\SS^2\sm \Lambda$ which have $\Lambda$ as their boundaries and the result follows from the previous corollary.
\end{proof}

This applies in particular to any Wada-type continuum (\ie a continuum $\Lambda$ whose complement has more than two connected components and the boundary of each such component is equal to $\Lambda$). Hence in the area-preserving setting, Wada-type continua can only appear as invariant sets if they have trivial dynamics.
This is in contrast with the dissipative setting, where these types of continua appear frequently with a rich dynamics; for example the Plikyn attractor or more generally any transitive hyperbolic attractor on the sphere.

As a second application, we obtain an elementary proof of one of the main results of \cite{kt-strictly} and \cite{kt-fully}, which states that in an orientable surface $S$, if $f\colon S\to S$ is a homeomorphism homotopic to the identity with an invariant open topological disk $U$ and $f$ has no wandering points (for example, if $f$ is area-preserving), and if the fixed point set of $f$ is inessential (\ie its inclusion in $S$ is homotopic to a point in $S$), then any open invariant topological disk must be homotopically bounded. The latter means that any lift of the disk to the universal covering space is relatively compact\textcolor{blue}{.} 
The proofs given in \cite{kt-strictly} and \cite{kt-fully} are involved and rely on the use of equivariant Brouwer theory \cite{lecalvez-equivariant} and maximal isotopies \cite{jaulent}. Using the results from this paper, one obtains a direct and simple proof. See \S\ref{sec:bounded} for more details.

Theorem \ref{th:3bd-free} is a special case of a more general result. In order to state it, let us introduce some definitions. 
Suppose $U$ is an open topological disk. A \emph{cross-cut} of $U$ is an open-ended simple arc $\alpha$ in $U$ which extends to a compact arc joining two different points of $\bd U$. \footnote{some definitions, as in \cite{kln}, allow these two points to coincide with an additional condition. Here it will suffice to use the more restrictive definition.}
Each cross-cut $\alpha$ separates $U$ into exactly two connected components called \emph{cross-sections} of $U$. 

Let $f\colon S\to S$ be an orientation-preserving homeomorphism such that $f(U)=U$.
Let $W\subset U$ be an open set, and denote by $\cl_U W$ the closure of $W$ in $U$. We say that $W$ is a \emph{positive boundary trapping region} for $f$ in $U$ if $f(\cl_U W) \subset W$ and $W$ is the union of a family of pairwise disjoint cross-sections of $U$ that is locally finite in $U$. A \emph{negative boundary trapping region} for $f$ is a positive boundary trapping region for $f^{-1}$, and a \emph{boundary trapping region} is a set which is either a positive or a negative boundary trapping region for $f$ in $U$.
We say that a loop $\gamma$ in $S$ is \emph{trapping} in $U$ if there is a boundary trapping region $W$ for $f$ in $U$ such that $\bd_U W \subset \gamma$, and otherwise we say that $\gamma$ is non-trapping.

This is the key result of this note:
\begin{theorem}\label{th:3bd-trap}
Suppose $f\colon \SS^2\to \SS^2$ is an orientation-preserving homeomorphism, and $B$ is a closed topological disk such that $f(B)\cap B=\emptyset$. If $B$ intersects three pairwise disjoint open $f$-invariant topological disks, then $\bd B$ is trapping in one of the three disks.
\end{theorem}

Note that if $U$ has a boundary trapping region $W$, then every point of $\bd_U W$ is wandering, so in particular a homeomorphism without wandering points cannot have any boundary trapping regions. Thus Theorem \ref{th:3bd-free} and its corollaries follow immediately from Theorem \ref{th:3bd-trap}.

\subsection{Maximal cross-cuts}

A key element for the proof of our main result is the following lemma. Its proof is essentially contained in \cite[Lemma 4.6]{kln}. In order to keep this article self-contained we present a proof here as well.

\begin{lemma} \label{lem:mccl} Let $S$ be an orientable surface, $f\colon S\to S$ an orientation-preserving homeomorphism, and $U$ be an invariant topological disk. Suppose that $B$ is a closed topological disk such that its interior intersects $\bd U$ and $f(B)\cap B=\emptyset$. If $\bd B$ is non-trapping, then there exists an arc $\sigma$ in $U$ joining a point $z\in B$ to $f(z)\in f(B)$ which intersects $B \cup f(B)$ only at its endpoints.
\end{lemma}

\begin{proof}
Consider the family $\mc C_0$ of all connected components of $\bd B\cap U$. Note that $\mc C_0$ is nonempty, since $B$ intersects $U$ and cannot contain $U$ entirely as $f(B)\cap B=\emptyset$.
Each element of $\mc C_0$ is a free (\ie disjoint from its image by $f$) cross-cut of $U$. 

If $\alpha$ is any free cross-cut of $U$, we write $U\sm \alpha = D(\alpha) \cup D'(\alpha)$ where $D'(\alpha)$ is the connected component of $U\sm \alpha$ containing $f(\alpha)$ and $D(\alpha)$ is the remaining component. 
One may easily verify that $f(D(\alpha)) = D(f(\alpha))$ and similarly for $D'(\alpha)$. 

We claim that for $\alpha\in \mc C_0$ one has $f(D(\alpha))\subset D'(\alpha)$ (so $D(\alpha)$ is free for $f$). Indeed, since $f(\alpha)$ is disjoint from $D(\alpha)$, the latter set is contained in a single connected component of $U\sm f(\alpha)= f(D(\alpha))\cup f(D'(\alpha))$. If $D(\alpha)$ intersects $f(D(\alpha))$ then $D(\alpha) \subset f(D(\alpha))$, which implies $\cl_U D(\alpha)\subset f(D(\alpha))$. But then $f^{-1}(\cl_U D(\alpha))\subset D(\alpha)$, so $D(\alpha)$ is a (negative) boundary trapping region contradicting our hypothesis.

Define a partial order among free cross-cuts of $U$ by writing $\alpha\prec \beta$ if $D(\alpha)\subset D(\beta)$. Let $\mc C_1 = \{f(\alpha):\alpha\in \mc C_0\}$, and $\mc C = \mc C_0\cup \mc C_1$. Note that for $\alpha\in \mc C$ one still has $f(D(\alpha)) \subset D'(\alpha)$, since we just showed this in the case that $\alpha\in \mc C_0$, and if $\alpha\in \mc C_1$ one has $\alpha = f(\alpha')$ for some $\alpha'\in \mc C_0$ so $f(D(\alpha)) = f(D(f(\alpha')) = f^2(D(\alpha')) \subset f(D'(\alpha')) = D'(f(\alpha'))$. 

Let us note that for each $c>0$ there are at most finitely many elements of $\mc C$ with diameter greater than $c$, since $\mc C$ consists of pairwise disjoint arcs in $\bd B \cup \bd f(B)$. As a consequence, the family $\mc C$ is locally finite in $U$. 

Denote by $\mc C^*$ the set of all elements of $\mc C$ which are maximal (in $\mc C$) with respect to $\prec$. We claim that for every $\alpha\in \mc C$ there exists $\alpha_*\in \mc C^*$ such that $\alpha\preceq \alpha_*$. In fact, one can show that there are finitely many elements $\alpha'\in \mc C$ such that $\alpha\prec \alpha'$. To see this, note that if $\alpha\prec \alpha'$ we have $f(D(\alpha))\subset f(D(\alpha')) \subset D'(\alpha')$ whereas $\alpha\subset D(\alpha')$, so $\alpha'$ separates $\alpha$ from $f(\alpha)$ in $U$. From this one may conclude that the diameter of $\alpha'$ is bounded below by some positive number $c$ which depends only on $\alpha$, so $\{\alpha' : \alpha\prec \alpha'\}$ is finite as claimed.

We now claim that there is an element $\alpha \in \mc C^*\cap \mc C_0$ such that $f(\alpha) \in \mc C^*$. Indeed, assume for a contradiction that this is not the case, and let $\alpha_0 \in \mc C^*$ be any element. Since $\alpha_0\in \mc C_0\cup \mc C_1$, we consider two possibilities: suppose first that $\alpha_0\in \mc C_0$. Since $f(\alpha_0)\notin \mc C^*$ by our assumption, there must exist $\alpha_1\in \mc C^*$ such that $f(\alpha_0)\prec \alpha_1$. Moreover, $\alpha_1$ cannot be in $\mc C_1$, since that would mean that $f^{-1}(\alpha_1)\in \mc C$ and $\alpha_0 \prec f^{-1}(\alpha_1)$ contradicting the maximality of $\alpha_0 \in \mc C^*$. Thus $\alpha_1\in \mc C_0$, and we may repeat this argument inductively to obtain an infinite sequence $(\alpha_i)_{i\in \N}$ of elements of $\mc C^*\cap \mc C_0$ such that $f(\alpha_i)\prec \alpha_{i+1}$. Note that $\{D(\alpha_i):i\in \N\}$ is a family of pairwise disjoint cross-sections (since each $\alpha_i$ is maximal) and it is locally finite in $U$ (which follows from the fact that $\{\alpha_i : i\in \N\}\subset \mc C$ is locally finite in $U$). Thus $W=\bigcup_{i\in \N} D(\alpha_i)$ is a (positive) boundary trapping region with $\bd_U W=\bigcup_{i\in\N}\alpha_i\subset \bd B$, contradicting our hypothesis.
Now suppose that $\alpha_0\in \mc C_1$. Then $f^{-1}(\alpha_0)\in \mc C_0$, and there must exist $\alpha_1\in \mc C^*$ such that $f^{-1}(\alpha_0)\prec \alpha_1$ (since otherwise $\alpha = f^{-1}(\alpha_0)$ would be such that both $\alpha$ and $f(\alpha)$ belong to $\mc C^*$ contradicting our assumption). Moreover, $\alpha_1\in \mc C_1$, since in the case that $\alpha_1\in \mc C_0$ one has $\alpha_0\prec f(\alpha_1)\in \mc C$ contradicting the maximality of $\alpha_0$. 
Thus by a similar argument we obtain a sequence $(\alpha_i)_{i\in \N}$ in $\mc C^*\cap \mc C_1$ such that $f^{-1}(\alpha_i)\prec \alpha_{i+1}$, and $W=\bigcup_{i\in \N} D(\alpha_i)$ is a (negative) boundary trapping region, and moreover $f^{-1}(W)$ is a negative boundary trapping region with $\bd_U W \subset \bd B$ contradicting our hypotheses. This proves our claim.

Finally, given $\alpha\in \mc C^*\cap \mc C_0$ such that $f(\alpha)\in \mc C^*$, we claim that there exists an arc $\sigma$ in $U$ joining a point $z\in \alpha$ to $f(z)\in f(\alpha)$ which is disjoint from all elemeents of $\mc C$ except at its endpoints.
To see this, let $K$ be the union of all elements of $\mc C$ except $\alpha$ and $f(\alpha)$, which is a closed subset of $U$. If the closed set $K$ separates $\alpha$ from $f(\alpha)$ in $U\simeq \R^2$ then some connected component $\beta$ of $K$ must separate $\alpha$ from $f(\alpha)$ (see for instance \cite[Theorem 14.3]{newman}). But the connected components of $K$ are elements of $\mc C$, so $\beta\in \mc C$, and the fact that $\beta$ separates $\alpha$ from $f(\alpha)$ implies that either $\alpha\prec \beta$ (contradicting the maximality of $\alpha$) or $f(\alpha)\prec \beta$ (contradicting the maximality of $f(\alpha)$. Thus $\alpha$ and $f(\alpha)$ are in the same connected component of $U\sm K$, which means that there is an arc $\sigma$ in $U\sm K$ joining any given $z\in \alpha$ to $f(z)\in f(\alpha)$. Of course $\sigma$ may be chosen so that it only intersects $\alpha$ and $f(\alpha)$ at its endpoints. Thus $\sigma$ only intersects $\bd B$ or $\bd f(B)$ at its endpoints, from which the required properties follow.
\end{proof}

\subsection{Proof of the main theorem}

Only Theorem \ref{th:3bd-trap} requires a proof, since the other results follow as explained in the introduction. The proof follows from Lemma \ref{lem:mccl} and a simple observation using the fact that $f$ preserves orientation.

Let $f$ and $B$ be as in the statement of Theorem \ref{th:3bd-trap}, and let $U_1, U_2, U_3$ be three disjoint $f$-invariant open topological disks such that $U_i\cap B\neq \emptyset$. Note that the interior of $B$ must intersect each $U_i$ as well. Assume for a contradiction that $\bd B$ is non-trapping in all three disks. Since the interior of $B$ intersects both $U_i$ and its complement, it must intersect $\bd U_i$. Applying Lemma \ref{lem:mccl} on each $U_i$ we know that there exists $z_i \in \bd B\cap U_i$ and an arc $\sigma_i\subset U_i$ joining $z_i$ to $z_i' := f(z_i)$ which is disjoint from $B\cup f(B)$ except at its endpoints. Let $\gamma = \bd B$, oriented positively so that $\inter B$ is locally on the left of $\gamma$, and let $\gamma' = f(\gamma)$. Since $f$ preserves orientation, $\inter f(B)$ is locally on the left of $f(\gamma)$ as well. Note that since $\sigma_i$ is disjoint from $\inter B$, this implies that  $\sigma_i$ is locally on the right of $\gamma$ near $z_i$ and on the right of $\gamma'$ near $z_i'$. 

By permuting the points if necessary, we may assume that $z_2$ lies in the positive subarc $\gamma_{13}$ of $\gamma$ joining $z_1$ to $z_3$. Since $f$ preserves orientation, $f(\gamma_{13})$ is the positive subarc of $\gamma'$ joining $z_1'$ to $z_3'$. In particular, the positive subarc $\gamma_{31}'$ of $\gamma'$ joining $z_3'$ to $z_1'$ does not contain $z_2'$. 
Let $\eta = \gamma_{13} * \sigma_3 * \gamma_{31}' * \sigma_1^{-1}$ (where $\sigma_1^{-1}$ denotes the arc $\sigma_1$ reversed). Then $\eta$ is a simple loop. Let $D$ be the component of the complement of $\eta$ which lies locally to the right of $\eta$. Since $\inter B$ is locally to the left of $\gamma_{13}$ and is disjoint from $\eta$, we have that $\inter B$ is disjoint from $D$, so its closure $B$ is also disjoint from $D$. For similar reasons, $f(B)$ is disjoint from $D$, and in particular $z_2'\notin D$. Moreover, since $z_2'\notin \eta$ we have $z_2'\notin \ol{D}$.

On the other hand, if $\dot{\sigma_2}$ denotes the arc $\sigma_2$ with its endpoints removed (which is disjoint from $\eta=\bd D$), we see that $\dot{\sigma_2}$ is locally on the left of $\gamma_{13}$ near $z_2$, and therefore $\dot{\sigma_2}\subset D$. Hence $\sigma_2\subset \ol{D}$, which contradicts the fact that $z_2'\notin \ol{D}$.
This proves Theorem \ref{th:3bd-trap}.\qed

\begin{figure}
\includegraphics[width=\textwidth]{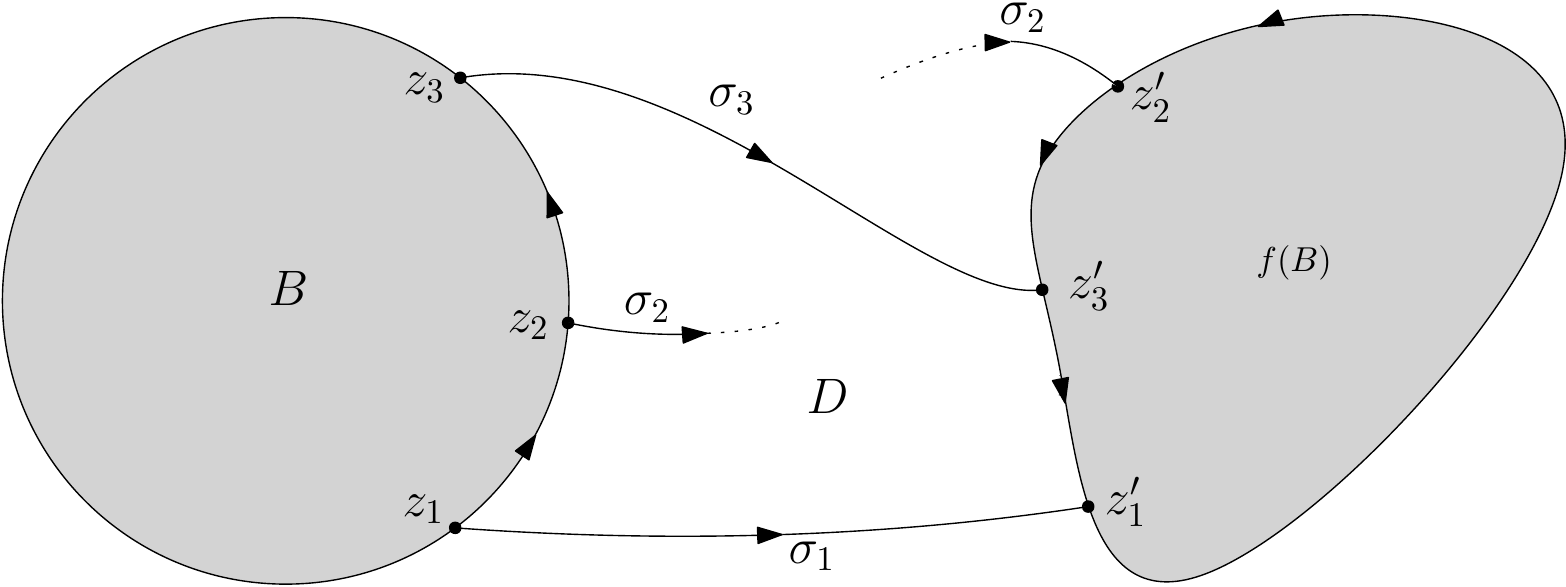}
\caption{Proof of main theorem}
\end{figure}

\subsection{An application: homotopically bounded disks}\label{sec:bounded}

Suppose that $S$ is a closed orientable surface, and $\pi\colon \til S\to S$ the universal covering map. We may endow $S$ with a metric of constant curvature, and $\til S$ with the lifted metric. For a set $X\subset \til S$ we denote by $\diam(X)$ its diameter. The \emph{covering diameter} of an open topological disk $U\subset S$ is defined as $\diamup(U) = \diam(\til U) \in \R^+_0 \cup \{\infty\}$, where $\til U$ is any connected component of $\pi^{-1}(U)$ (and it is independent from the choice of $\til U$). When $\diamup(U)< \infty$, we say that $U$ is homotopically bounded. 
The next result was one of the key theorems in \cite{kt-strictly} and \cite{kt-fully}. Using our main theorem, we are able to give a very direct proof of it. 
\begin{theorem} 
Let $f\colon S\to S$ be a homeomorphism homotopic to the identity of a closed orientable surface such that $\fix(f)$ is inessential. Then there exists $M>0$ such that any open $f$-invariant topological disk $U$ without boundary trapping regions satisfies $\diamup(U)\leq M$.
\end{theorem}
\begin{proof}
Denote by $G$ the group of deck transformations of $\pi$ (which are all isometries of $\til S$). Since $f$ is isotopic to the identity, there exists a lift $\til f_0$ of $f$ which commutes with every $T\in G$. Fix a connected component $\til U$ of $\pi^{-1}(U)$. Since $U$ is invariant there exists $T_0\in G$ such that $\til f_0(\til U)=T_0\til U$. Consider $\til f = T_0^{-1}\til f_0$, so that $\til f(\til U) = \til U$.

If $S=\SS^2$ there is nothing to be done. Suppose now that $S=\T^2$. In that case, since $G$ is abelian, $\til f$ also commutes with every element of $G$, and in particular for each $T\in G$ one has $\til f(T\til U) = T\til U$. The fact that $\fix(f)$ is inessential means that one may find a compact connected subset $Q$ of $\til S$, such that $\pi(Q) = S$ and $\fix(\til f)\cap \bd Q=\emptyset$ (see \cite[Remark 2.1]{kt-fully}). Since $\bd Q$ is compact and has no fixed points of $\til f$, one may cover $\bd U$ by a finite family $B_1, \dots, B_m$ of closed disks such that $\til f(B_i)\cap B_i=\emptyset$ for each $i$. If $\til U$ has diameter greater than $M:=(2m+1)\diam(Q)$, then $\til U$ must intersect $T(\bd Q)$ for at least $2m+1$ different values of $T\in G$. As a consequence, there exists $i\in \{1,\dots, m\}$ such that $B_i$ intersects $T^{-1}\til U$ for at least three different values of $T\in G$. These are three pairwise disjoint $\til f$-invariant topological disks, so by Theorem \ref{th:3bd-trap} there is $T\in G$ such that $T^{-1}\til U$ has a boundary trapping region $\til W$ for $\til f$. Since $U$ is a topological disk, $\pi|_{T^{-1}\til U}$ is a homeomorphism onto $U$, and it follows that $\pi(W)$ is a boundary trapping region in $U$ for $f$ contradicting our hypothesis.

Finally, if $S$ is a hyperbolic surface and $T_0=\id$ (which means that $\til f$ commutes with every element of $G$) the same proof used for $\T^2$ applies. Thus it remains to consider the case where $T_0\neq \id$. In that case we know that $\til f$ commutes with $T_0$ but not necessarily with other elements of $G$. However, this case can be reduced to the case of $\T^2$ as in the proof of \cite[Proposition 4.8]{kt-fully}. Since the argument is identical, we omit these details and refer the reader to \cite{kt-fully}.
\end{proof}

\subsection*{Acknowledgements}
A. Koropecki was partially supported by the German Research Council (Mercator fellowship, DFG-grant OE 538/9-1), as well as FAPERJ-Brasil and CNPq-Brasil.
F. A. Tal was partially supported by the Alexander Von Humboldt foundation and by FAPESP, CNPq and CAPES. Both would like to thank FSU-Jena and T. J\"ager for the hospitality during the writing of this paper. 

\bibliographystyle{amsalpha}

\end{document}